\documentclass[12pt]{amsart}

\usepackage{amsfonts}
\usepackage{amssymb, latexsym, amsthm, amsmath, verbatim}
\usepackage{mathrsfs}
\usepackage{multirow}
\usepackage{indentfirst}
\usepackage{float}
\usepackage{enumerate}
\usepackage{booktabs}    % 专业表格样式
\usepackage{xurl}

\newcommand{\QQ}
{\mathbb{Q}} % rational

\numberwithin{equation}{section}
\theoremstyle{definition}
\newtheorem{Thm}{Theorem}
\newtheorem{Prop}[equation]{Proposition}
\newtheorem{Cor}[equation]{Corollary}
\newtheorem{Lem}[equation]{Lemma}

\newtheorem{Rmk}[equation]{Remark}

\makeatletter
\def\imod#1{\allowbreak\mkern5mu{\operator@font mod}\,\,#1}
\makeatother

\usepackage{color}

\definecolor{blue}{rgb}{0,0,1}
\definecolor{red}{rgb}{1,0,0}
\definecolor{green}{rgb}{0,.6,.2}
\definecolor{purple}{rgb}{1,0,1}

\long\def\red#1\endred{{\color{red}#1}}
\long\def\blue#1\endblue{{\color{blue}#1}}
\long\def\purple#1\endpurple{{\color{purple}#1}}
\long\def\green#1\endgreen{{\color{green}#1}}
\allowdisplaybreaks[4]

\begin{document}

\title[]{Explicit Hecke eigenform product identities for Hilbert modular forms}
\author{Zeping Hao and Chao Qin and Yang Zhou$^{\star}$}
\address{Hua Loo-Keng Center for Mathematical Sciences, Academy of Mathematics and System Sciences, Chinese Academy of Sciences, 100190 Beijing, P.R.China}
\email{zepinghao@amss.ac.cn}
\address{College of Mathematics Science,
Harbin Engineering University,
150001 Harbin, P.R.China}
\email{qinchao@hrbeu.edu.cn}
\address{College of Mathematics Science,
Harbin Engineering University,
150001 Harbin, P.R.China}
\email{yang.zhou@hrbeu.edu.cn}
\date{}
\subjclass[2020]{Primary: 11F41, 11F30.}
\thanks{$^{\star}$ corresponding author}

\keywords{Hecke eigenform, Hilbert modular form, Product identity.}

\begin{abstract}
Let $F$ be a totally real number field, and $g,f,h$ be Hilbert modular forms over $F$ that are Hecke eigenforms satisfying $g=f\cdot h$. Under the grand Riemann hypothesis, we characterize such product identities among all real  quadratic fields of narrow class number one, proving they occur only for $F=\mathbb Q(\sqrt{5})$, with precisely two such identities. We also shed some light on the general totally real case by showing that no such identity exists when both $f$ and $h$ are Eisenstein series of distinct weights. 
% For all general totally real number fields of degree $3$ or higher, we prove no such identity exists when both $f$ and $h$ are Eisenstein series of distinct weights. 
\end{abstract}
\maketitle
%\tableofcontents
\newcommand{\Z}{{\mathbb Z}} % for integers
\newcommand{\Q}{{\mathbb Q}} % for rational
\newcommand{\R}{{\mathbb R}}
\newcommand{\C}{{\mathbb C}}
\newcommand{\N}{{\mathbb N}}
\newcommand{\h}{{\mathbb H}}
\newcommand{\SR}{{\mathrm{SL}_2(\R)}}

\section{Introduction and Statement of the Main Theorem}
Nearly three decades ago,  William Duke proposed a question: When is the product of two Hecke eigenforms an eigenform? Duke \cite{Duke1997product} and Ghate  \cite{Ghate2000Eisenstein} independently discovered exactly $16$ eigenform product identities $g=f\cdot h$ for $\mathrm{SL}_2(\mathbb{Z})$, all of which hold trivially for dimension reasons. Later, Johnson \cite{Johnson2013eigenforms} extended this result to $61$ eigenform product identities over all levels, all weights and Nebentypus, some of which hold non-trivially. Recently, Joshi and Zhang \cite{Joshi-Zhang2019Hilbert} generalized this question to Hilbert modular forms over real quadratic fields, proving the finiteness of eigenform product identities $g=f\cdot h$  among full-level Hecke eigenforms of weight $2$ or greater. They showed that there exist two eigenform product identities when $F=\mathbb{Q}(\sqrt{5})$. You and Zhang \cite{You-Zhang2021Hilbert}  further established the finiteness of eigenform product identities over all totally real number fields of fixed degree $n$. 

In this paper, we enumerate all such product identities over all quadratic fields with narrow class number one and all product identities of two distinct-weight Eisenstein series over all totally real number fields of degree $3$ or greater. Our results extend  those  in \cite{Joshi-Zhang2019Hilbert} and \cite{You-Zhang2021Hilbert}, in the sense that we exhaust the possibilities  of all such identities in the described setups. More precisely,
we prove the following theorems.
\begin{Thm}\label{Thm1}
Over all real quadratic fields $F$ of narrow class number one and full-level Hecke eigenforms of parallel weights, eigenform product identities exist only for $F=\mathbb Q(\sqrt 5)$ under the grand Riemann hypothesis, with explicit identities
$$E_4=60E_2^2,\quad h_8=120E_2\cdot h_6,$$ 
as established in \cite[Theorem 7.4]{Joshi-Zhang2019Hilbert}.
Here $E_k$ is the normalized Hecke eigenform of weight $k$ in the Eisenstein subspace, and  $h_6, h_8$ are the unique normalized cuspidal eigenforms of weights $6$ and $8$ respectively. More precisely, over all such fields with discriminant $D>5$, we have  
\begin{enumerate}
\item No eigenform product identity $g = f \cdot h$ exists when $g,f,h$ are Eisenstein series of weight $2$ or greater, and $f,h$ are normalized.
\item No eigenform product identity $g = f \cdot h$ exists where $g$ is a Hecke eigenform, with one of $f,h$ a normalized Eisenstein series of weight $4$ or greater and the other a normalized cuspidal eigenform. Under the grand Riemann hypothesis, this nonexistence extends to identities involving weight $2$ normalized Eisenstein series paired with any normalized cuspidal eigenforms.
\end{enumerate}
\end{Thm}
It is immediate that the product of two cuspidal eigenforms cannot be an eigenform since its  Fourier coefficient at the unit ideal vanishes. Therefore Theorem \ref{Thm1} separates into two cases: either one of $f,h$ is cuspidal or both are Eisenstein series. Inspired by the method of \cite{Joshi-Zhang2019Hilbert}, we use the relations between Fourier coefficients at small primes or at powers of small primes in the eigenform product identity $g=f\cdot h$
 to obtain effective bounds for such identities. As the numerical values of the Fourier coefficients of eigenforms are difficult to compute when the narrow class number exceeds one, we restrict our attention to the case where the narrow class number equals one.

 For the first case of Theorem \ref{Thm1}, we first bound the discriminants of real quadratic fields in which an eigenform product identity may occur. For each such real quadratic field, we then bound the possible weights for which  the eigenform product identity can hold. Finally, by comparing Fourier coefficients of Eisenstein series using SageMath, we obtain a complete list of eigenform product identities.
 
 For the second case of Theorem \ref{Thm1}, effective bounds on the discriminant and the weight can be derived in a similar manner. However, computing the Fourier coefficients of eigenforms in higher-dimensional cusp form spaces is computationally difficult. We therefore apply the main theorem of \cite{Rankin-Cohen} together with the grand Riemann hypothesis (see page 60 of \cite{Borwein2008Riemann}) to conclude that eigenform product identities hold trivially for dimension reasons (see Theorem \ref{RC2024} for details). This allows us to avoid the computation of Fourier coefficients in higher dimensions. 
 \begin{Rmk}
Note that only the second part of Theorem \ref{Thm1} depends on the grand Riemann hypothesis, which predicts that all zeros of a normalized automorphic $L$-function with $0<\mathrm{Re}(s)<1$ lie on the line $\mathrm{Re}(s)=\frac{1}{2}$. The main theorem of \cite{Rankin-Cohen} does not cover the case $g = E_2 \cdot h$ with $h$ cuspidal, since the relevant point does not lie in the region of absolute convergence of the Rankin–Selberg $L$-function. However, under the grand Riemann hypothesis, this case can be handled.  
 \end{Rmk}

In essence, under the grand Riemann hypothesis, our result shows that Hecke eigenform product identities occur only when $F=\QQ (\sqrt{5})$ for dimension reasons. Explicit dimension formula involving the Dedekind zeta values implies that the dimension of the relevant space grows with the discriminant. As 
$F=\QQ (\sqrt{5})$ possesses the minimal discriminant among totally real quadratic fields, it yields the minimal dimension required for such identities to hold trivially. Indeed, we establish relations among special values of the Dedekind zeta functions, with estimates involving factors of the discriminant, and increasing the discriminant invalidates these relations. Similarly, we can generalize our results to arbitrary totally real number fields of degree $3$ or greater, concerning the case of unequal-weight Eisenstein series.  

\begin{Thm}\label{Thm2}
Over all totally real number fields of degree $n>2$ and all full-level Hecke eigenforms of weight $2$ or greater, no eigenform product identity  $g=f\cdot h$ exists, where $g,f,h$ are Hecke eigenforms in Eisenstein subspaces with $f,h$ being normalized and having distinct weights.
\end{Thm}

We expect  our methods to  apply equally to other cases, such as equal-weight Eisenstein series and  cusp form-Eisenstein series identities, once appropriate data including  special values of Hecke $L$-series and explicit dimension formulas for spaces of cusp forms  become available. We hope to address this in future work.

The layout of this paper is as follows. In Section \ref{Section 2}, we set up the notations and introduce the basics. Section \ref{Section 3} is dedicated to the first case of Theorem \ref{Thm1}, establishing Propositions \ref{Lemgeq41}, \ref{lemleq41}, and \ref{prop3.6}. Section \ref{Section 4} completes the proof of Theorem \ref{Thm1} by handling the second case via Propositions \ref{(2)Inert} and \ref{(2)Non-inert}. Finally, Section \ref{Section 5} provides the proof of Theorem \ref{Thm2}.

\section{Preliminaries}\label{Section 2}
Let $F$ be a totally real number field of degree $n$ over $\QQ$, with ring of integers $\mathcal{O}$, different $\mathfrak{d}$, discriminant $D$, class number $h$, and narrow class number $h^+$. Denote its group of positive units by $\mathcal{O}^+$ and that of totally positive units by $\mathcal{O}^{\times+}$. In this paper, we are only interested in the Hilbert modular group of full level
$$
\Gamma_F=\Gamma_{0}(\mathcal{O},\mathcal{O})=\left\{\gamma=\left(\begin{array}{ll}
a & b \\
c & d
\end{array}\right) \in\left(\begin{array}{cc}
\mathcal{O} & \mathfrak{d}^{-1} \\
 \mathfrak{d} & \mathcal{O}
\end{array}\right): \mathrm{det}(\gamma) \in \mathcal{O}^{\times+}\right\},
$$ which can be embedded into $\mathrm{GL}^+_{2}(\mathbb{R})^n$ by
$$\left(\begin{matrix}
a & b \\
c & d
\end{matrix}\right) \mapsto\left(\left(\begin{matrix}
a_1 & b_1 \\
c_1 & d_1
\end{matrix}\right), \cdots,\left(\begin{matrix}
a_n & b_n \\
c_n & d_n
\end{matrix}\right)\right),$$
where $a\mapsto{(a_1,\cdots, a_n)}$ gives the embedding $F\subset F\otimes_{\mathbb Q}\mathbb R$. 
Let $\mathbb{H}$ be the complex upper half plane. A Hilbert modular form of parallel weight $k\in\mathbb{Z}$ for $\Gamma_F$ is a holomorphic function $f$ on $\mathbb{H}^{n}$ such that 
$$
\left(f|_{k}\gamma\right)(z)=\mathrm{det}(\gamma)^{k/2}j(\gamma,z)^{-k} f(\gamma z)=f(z), \text{ for any } \gamma\in \Gamma_F,
$$
where $z=\left(z_{1},\cdots,z_{n}\right) \in \mathbb H^{n}$,  $\mathrm{det}(\gamma)=(\mathrm{det}(\gamma_1),\cdots,\mathrm{det}(\gamma_n))$ and  $$\gamma z=\left(\frac{a_{1} z_{1}+b_{1}}{c_{1} z_{1}+d_{1}}, \cdots, \frac{a_{n} z_{n}+b_{n}}{c_{n} z_{n}+d_{n}}\right),\quad j(\gamma,z)=\left(c_{1} z_{1}+d_{1},\cdots,c_{n} z_{n}+d_{n}\right).$$
Denote by $M_{k}\left(\Gamma_F\right)$ and  $S_{k}\left(\Gamma_F\right)$  the space of Hilbert modular forms and that of cusp forms of weight $k\in \mathbb{N}$ for $\Gamma_F$ respectively. Here $\mathbb{N}$ is the set of natural numbers and $0\in\mathbb N$.  Let $\mathcal{E}_k(\Gamma_F)$ be the Eisenstein subspace, orthogonal to $S_k(\Gamma_F)$ in $M_k(\Gamma_F)$. Every $f \in M_{k}\left(\Gamma_F\right)$ has a unique Fourier expansion at the cusp of the form
$$f=\sum_{\nu\in \mathcal{O}}c_{\nu}(f)\mathrm{exp}\left(2\pi i \sum_{i=1}^{n} \nu_{i} z_{i}\right).$$

Now, let us briefly review the Hecke theory for Hilbert modular forms, and for simplicity, we restrict to $h^+=1$. For any non-zero integral ideal $\mathfrak{n}=(\nu)$ with $\nu\in\mathcal{O}$, the $\mathfrak n$-th Fourier coefficient of $f$ is defined to be $c(\mathfrak{n}, f)=c_{\nu}(f)$ (see \cite[Eq (2.24)]{Shimura1978Hilbert}), and the constant term of $f$ is denoted by $c_0(f)$. The $\mathfrak n$-th Hecke operator $T_\mathfrak{n}$ acts on $M_{k}\left(\Gamma_F\right)$, preserving  $S_{k}\left(\Gamma_F\right)$ and  $\mathcal E_{k}\left(\Gamma_F\right)$. A Hecke eigenform is a non-zero common eigenfunction for all Hecke operators $T_{\mathfrak{n}}$. Since the Hecke operators $T_\mathfrak{n}$ on $S_k(\Gamma_F)$ are self-adjoint and mutually commute, they admit a basis of eigenforms.  Any Hecke eigenform $f$ satisfies $c(\mathcal{O},f)\neq 0$, and is called normalized if $c(\mathcal{O},f)=1$ (see  \cite[p.650]{Shimura1978Hilbert}). 
For a normalized eigenform $f$, the $T_\mathfrak{n}$-eigenvalue is exactly $c(\mathfrak{n},f)$, and if $\mathfrak p$ is a prime ideal, then 
\begin{align}\label{primeCoe}
c(\mathfrak{p}^{j+1},f)&=c(\mathfrak{p}^{j},f)c(\mathfrak{p},f)-N(\mathfrak{p})^{k-1}c(\mathfrak{p}^{j-1},f),\quad j=1,2,\cdots.
\end{align}

If $f\in S_k(\Gamma_F)$ is a normalized Hecke eigenform, then by the Ramanujan conjecture proved in \cite[Theorem 1]{Blasius2006Ramanujan}, for any prime ideal $\mathfrak{p}$,
\begin{align}\label{cuspFourierBound}
|c(\mathfrak p, f)|\leq
2N(\mathfrak p)^{\frac{k-1}{2}}.
\end{align}
If $F$ has narrow class number one, the dimension of $\mathcal{E}_{k}(\Gamma_F)$ is equal to $1$. Denote its  normalized eigenform by $E_k$, whose Fourier coefficients satisfy   \cite{Dasgupta2011Hilbert, Shimura1978Hilbert}:
\begin{equation}\label{EisensteinCoe}
c(\mathfrak n,E_k)=\sum\nolimits_{\mathfrak r\mid \mathfrak n}N(\mathfrak r)^{k-1},\quad c_0(E_k)=2^{-n}\zeta_F(1-k)
\end{equation}
 for any non-zero integral ideal $\mathfrak n$, with  $\zeta_F(k)$ the Dedekind zeta function of $F$. Note that $\zeta_F(k)$ satisfies the following bounds (see \cite[Eq (2.3)]{Joshi-Zhang2019Hilbert}) 
\begin{equation}\label{Zeta-bound}
\begin{aligned}
\frac{2}{\pi} \left( \frac{D}{4\pi^2} \right)^{k - \frac{1}{2}} \Gamma(k)^2 
\frac{\zeta(4k)}{\zeta^2(k)} 
\leq |\zeta_F(1-k)| \leq \frac{2}{\pi} \left( \frac{D}{4\pi^2} \right)^{k - \frac{1}{2}} \Gamma(k)^2\zeta^2(k),
\end{aligned}
\end{equation}
where $\zeta(k)$ is the Riemann zeta function and we further have 
\begin{align}\label{zeta_bound}
\frac{72}{\pi^5} \left(\frac{D}{4\pi^2}\right)^{k-\frac{1}{2}} \Gamma(k)^2 &\leq |\zeta_F(1-k)| \leq \frac{\pi^3}{18} \left(\frac{D}{4\pi^2}\right)^{k-\frac{1}{2}} \Gamma(k)^2,
\end{align} since $1<\zeta(k)\leq\zeta(2)=\frac{\pi^2}{6}$.

%\noindent {\bf Acknowledgment}: The authors thank Yichao Zhang for his interest in our work and useful comments. The second author is supported by the National Natural Science Foundation of China under Grant No. 12001546, Heilongjiang Province under Grant No. 3236330122.

\section{The case of two Eisenstein series}\label{Section 3}
In this section, let $F$ be a real quadratic field with narrow class number one and $D > 5$. We prove the first part of Theorem \ref{Thm1} by contradiction,  assuming that $f$ and $h$ are normalized Hecke eigenforms with $c_0(f)c_0(h) \neq 0$, and that their product $g = f \cdot h$ is also a Hecke eigenform. 
The mechanism for the contradiction relies on comparing the Fourier coefficients of both sides of this identity $g=f\cdot h$, which yields relations involving  special values of the Dedekind zeta function
 that are incompatible with analytic bounds derived from their growth properties. 

Since the Eisenstein subspace is non-trivial only if the weight is even, we assume $k_1,k_2\geq2$ are even in this case.  Put $f=E_{k_1}$ and $h=E_{k_2}$. Since $g$ is a Hecke eigenform, it must be proportional to $E_{k_1+k_2}$. By  comparing the term corresponding to the trivial ideal $\mathcal O$
 in $g=E_{k_1}\cdot E_{k_2}$, we have $$g=(c_0(f)+c_0(h))E_{k_1+k_2}.$$ Then from the constant terms, we have 
$$(c_0(f)+c_0(h))\cdot c_0(E_{k_1+k_2})=c_0(f)\cdot c_0(h).$$
Hence, by \eqref{EisensteinCoe}, we obtain the key equation
\begin{equation}\label{constant}
1= (\zeta_F(1 - k_1) + \zeta_F(1 - k_2)) \frac{\zeta_F(1 - k_1 - k_2)}{\zeta_F(1 - k_1)\zeta_F(1 - k_2)}.
\end{equation}

We treat the unequal-weight and the equal-weight cases separately. In the unequal-weight case, the bounds \eqref{Zeta-bound} for the Dedekind zeta function suffice to show that \eqref{constant} fails for  $D>5$. However, the equal-weight case  additionally requires the relations between Fourier coefficients at small primes to establish effective bounds for the eigenform product identities.
\subsection{The unequal-weight case} Let  $f=E_{k_1}$, $h=E_{k_2}$ and assume $k_1>k_2$ without loss of generality. As outlined above, we show that \eqref{constant} fails by establishing a lower bound on the right-hand side using \eqref{Zeta-bound}. Put 
\begin{align*}
&C(D,k_1,k_2)\\
=&\frac{\zeta(4(k_1+k_2))}{\zeta(k_1+k_2)^2\zeta(k_1)^2}  \left( \frac{D}{4\pi^2} \right)^{k_2} \frac{\Gamma(k_1 + k_2)^2}{\Gamma(k_1)^2} \left| \frac{\zeta(4k_1)}{\zeta(k_1)^2\zeta(k_2)^2}\left( \frac{D}{4\pi^2} \right)^{k_1 - k_2}\frac{\Gamma(k_1)^2}{\Gamma(k_2)^2} - 1 \right|\\
\geq &\frac{\zeta(4(k_1+k_2))}{\zeta(k_1+k_2)^2\zeta(k_1)^2}  \left( \frac{D\cdot k_1^{2}}{4\pi^2} \right)^{k_2}\left| \frac{\zeta(4k_1)}{\zeta(k_1)^2\zeta(k_2)^2}\left( \frac{D}{4\pi^2} \right)^{k_1 - k_2}\frac{\Gamma(k_1)^2}{\Gamma(k_2)^2} - 1 \right|,\end{align*}
since $\Gamma(k_1)/\Gamma(k_2)\geq k_2^{k_1-k_2}$.
 By  \eqref{Zeta-bound} and \eqref{constant}, the  identity $g=E_{k_1}\cdot E_{k_2}$ must satisfy  \begin{equation}\label{EisensteinConstant} 
\begin{aligned}
1 &= \left| (\zeta_F(1 - k_1) + \zeta_F(1 - k_2)) \frac{\zeta_F(1 - k_1 - k_2)}{\zeta_F(1 - k_1)\zeta_F(1 - k_2)} \right|\\
&\geq\big||\zeta_F(1-k_1)|-|\zeta_F(1-k_2)|\big|\cdot\left|\frac{\zeta_F(1-k_1-k_2)}{\zeta_F(1-k_1)\zeta_F(1-k_2)}\right|\\
&\geq C(D,k_1,k_2),
\end{aligned}
\end{equation}
which leads to the contradiction proved in the following propositions. 

\begin{Prop}\label{Lemgeq41}
There is no eigenform product identity $g=E_{k_1}\cdot E_{k_2}$ over all real quadratic fields with narrow class number $1$ and  $D\geq41$.
\end{Prop}
\begin{proof}
By \eqref{EisensteinConstant}, it suffices to prove $C(D,k_1,k_2)$ is always greater than $1$ when $D\geq41>4\pi^2$ and $2\leq k_2< k_2+2\leq k_1$. 
Note that $\Gamma(k_1)/\Gamma(k_2)\geq k_2^{k_1-k_2}$, we have  $$\frac{\zeta(4k_1)}{\zeta(k_1)^2\zeta(k_2)^2}\left( \frac{D}{4\pi^2} \right)^{k_1 - k_2}\frac{\Gamma(k_1)^2}{\Gamma(k_2)^2}\geq\frac{\zeta(4k_1)}{\zeta(k_1)^2\zeta(k_2)^2}\left( \frac{Dk_2^2}{4\pi^2} \right)^{k_1 - k_2}>\frac{291600}{\pi^{12}}\cdot 2^{4}>1,$$ since $1<\zeta(k)\leq\zeta(2)=\frac{\pi^2}{6}$ and $\zeta(4)=\frac{\pi^4}{90}$.
It follows that
\begin{equation*}
\begin{aligned}
&C(D,k_1,k_2)\\
\geq&\frac{\zeta(4(k_1+k_2))}{\zeta(k_1+k_2)^2\zeta(k_1)^2}  \left( \frac{D}{4\pi^2} \right)^{k_2} \frac{\Gamma(k_1 + k_2)^2}{\Gamma(k_1)^2} \left( \frac{\zeta(4k_1)}{\zeta(k_1)^2\zeta(k_2)^2}\left( \frac{Dk_2^2}{4\pi^2} \right)^{k_1 - k_2} - 1 \right)\\
>& \frac{291600}{\pi^{12}}\left( \frac{41\cdot 4^2}{4\pi^2} \right)^{2}\left( \frac{291600}{\pi^{12}}\left( \frac{41\cdot 2^2}{4\pi^2} \right)^{2} - 1 \right)>1
\end{aligned}
\end{equation*}
always holds for $k_1>k_2\geq2$ and $D\geq41$, which contradicts \eqref{EisensteinConstant}. Hence, there is no eigenform product identity when $D\geq41$.
\end{proof}

Now we only need to consider the case where $D\in\{8,13,17,29,37\}$, since $F$ has narrow class number one.
\begin{Prop}\label{lemleq41}
There is no eigenform product identity $g=E_{k_1}\cdot E_{k_2}$ over all real quadratic fields with  $D\in\{8,13,17,29,37\}$.
\end{Prop}
\begin{proof}
Following the lines of Proposition \ref{Lemgeq41}, we give a detailed proof for $D=8$. We first consider $k_2\geq 4$ and have    $$\frac{\zeta(4k_1)}{\zeta(k_1)^2\zeta(k_2)^2}\left( \frac{8\cdot k_2^2}{4\pi^2} \right)^{k_1 - k_2}> \frac{291600}{\pi^{12}} \left( \frac{8\cdot 4^2}{4\pi^2} \right)^{2}>1.$$
It follows that   
\begin{align*}
C(8,k_1,k_2)> \frac{291600}{\pi^{12}}\left( \frac{8\cdot 6^2}{4\pi^2} \right)^{4} \left( \frac{291600}{\pi^{12}}\left( \frac{8\cdot 4^2}{4\pi^2} \right)^{2} - 1 \right)>1.
\end{align*}
 Hence, there is no eigenform product identity when $D=8$ and $k_1>k_2\geq4$.

 Now we only need to consider the case $k_2=2$ and determine the minimal value of $$\left| \frac{\zeta(4k_1)}{\zeta(k_1)^2\zeta(2)^2}\left( \frac{8}{4\pi^2} \right)^{k_1 -2}\frac{\Gamma(k_1)^2}{\Gamma(2)^2} - 1 \right|.$$ Note that the function $$f(k_1)=\left( \frac{8}{4\pi^2} \right)^{k_1 -2}\Gamma(k_1)^2$$ increases with respect to $k_1$ over $S = \{4 + 2k \mid k \in \mathbb{N}\}$, as for $k_1\geq 4$, $$\frac{f(k_1+1)}{f(k)}=\frac{2k_1^2}{\pi^2}>1.$$
Given that  $f(4)\geq 1.478$ and $f(6)\geq 24.281$, 
the inequality
$$\frac{\zeta(4k_1)}{\zeta(k_1)^2\zeta(2)^2}\left( \frac{8}{4\pi^2} \right)^{k_1 - 2}\frac{\Gamma(k_1)^2}{\Gamma(2)^2}>\frac{291600}{\pi^{12}}\cdot f(6)>1$$ holds for all $k_1\geq6$. It follows that for $k_1\geq 6$,
\begin{align*}
    C(8,k_1,2)>
        \frac{291600}{\pi^{12}}\left( \frac{8\cdot6^2}{4\pi^2} \right)^{2}\left(\frac{291600}{\pi^{12}}\cdot f(6)-1\right)>1, 
\end{align*}
which contradicts \eqref{EisensteinConstant}.  For the remaining triple $(D,k_1,k_2)=(8,4,2)$, a numerical evaluation using SageMath yields $C(8,k_1,k_2)\approx7.2291$, contradicting  \eqref{EisensteinConstant}.  Hence no eigenform product identity exists for $D=8$.
The proofs for  other cases are analogous to  the   $D=8$ case but easier, we therefore leave the details to the reader.
\end{proof}
\begin{Rmk}
The relevant data associated with the computation of 
$C(D,k_1,k_2)$
 and the proof of Proposition \ref{lemleq41} are available on  GitHub \cite{Code}.
\end{Rmk}
By Proposition \ref{Lemgeq41} and Proposition \ref{lemleq41}, no eigenform product identity of the form $g=E_{k_{1}}\cdot E_{k_{2}}$ with $k_{1}\neq k_{2}$ exists over all quadratic fields with narrow class number one and $D>5$. 
\subsection{The equal-weight case} Now we consider the equal-weight case and put $f=h=E_k$.  In this case, \eqref{constant} and \eqref{Zeta-bound} are insufficient to derive a contradiction. Therefore, we must examine the Fourier coefficients at small primes. 

If $(2)$ is not inert, we may assume without loss of generality that 
$(2)=\mathfrak p^2$ or $(2)=\mathfrak p \mathfrak p^{\prime}$, where $\mathfrak p$
 is a prime ideal. Comparing the Fourier coefficients at 
$\mathfrak p$
 on both sides of $g=f\cdot h $ yields 
\begin{align}\label{equalweight}
 (2^{2k-1}-2^{k-1})/\zeta_F(1-2k)=0,   
\end{align}
which corresponds to a special case of \cite[Eqs.~(5.4) and (5.5)]{Joshi-Zhang2019Hilbert}. Since \eqref{equalweight} is impossible for $k\geq 2$, eigenform product identity can only occur  when $(2)$ is inert. 

Put $$C(D,k)=\left(\frac{108}{ \pi^6}\right)^{2}  \cdot D^{\frac{1}{2}}  \cdot k.$$
If $(2)$ is inert, eigenform product identity must satisfy  \cite[Eq (5.3)]{Joshi-Zhang2019Hilbert}, that is 
$$\frac{4^{2k - 1} - 4^{k - 1}}{\zeta_F(1 - 2k)} = \frac{4}{\zeta_F(1 - k) \zeta_F(1 - k)}.$$
Combining this with \eqref{zeta_bound},
eigenform product identity should satisfy
\begin{equation}\label{EE}
\begin{aligned}
1 \geq 4^{1-2 k}\left(4^{2 k-1}-4^{k-1}\right)\geq 4^{2-2 k} \frac{\frac{72}{\pi^5}\left(\frac{D}{4 \pi^2}\right)^{2 k-\frac{1}{2}} \Gamma\left(2 k\right)^2}{\left(\frac{\pi^3}{18}\right)^2\left(\frac{D}{4 \pi^2}\right)^{2 k-1} \Gamma\left(k\right)^2 \Gamma\left(k\right)^2}\geq C(D,k)
\end{aligned}
\end{equation}
by the Stirling's bound on the binomial coefficients
$\binom{2n}{n} \geq n^{-\frac{1}{2}}2^{2n-1}$. 
\begin{Prop}\label{prop3.6}
There is no eigenform product identity $g=E_{k}\cdot E_{k}$ over all real quadratic fields with narrow class number one and $D>5$.
\end{Prop}
\begin{proof}
When $(2)$ is not inert, \eqref{equalweight} directly proves the result. Hence we only need to consider the case when $(2)$ is inert.
Note that 
\( C(D,k) \) increases with both $D$ and $k$. We first fix $D$ at its minimum value $13$, yielding $k\leq 20$ by \eqref{EE}. For each  such $k$, we determine the maximum $D$ via \eqref{EE} (see Table \ref{tab:p-max-simple}). Exhaustive verification of equality \eqref{constant} is carried out by evaluating the rational values of $\zeta_F$
 at negative integers via generalized Bernoulli numbers for all resulting 
$(D,k)$ pairs. We find that no eigenform product identities exist in this case. This finishes the proof.
\end{proof}
\begin{table}[ht]
    \centering
    \caption{Maximal  possible $D$ for weight $k$}
    \label{tab:p-max-simple}
    \begin{tabular}{|c|c|c|c|c|c|c|c|c|c|c|}
    \hline 
         $k$&2  &4  &6  &8  &10  &12  &14  &16  &18 &20\\\hline
         \text{Maximal $D$} & 1549  &389  &173  &61 &61 &37  &29  &13  &13  &13 \\
         \hline
    \end{tabular}
\end{table}
\begin{Rmk}
The source code used to generate Table \ref{tab:p-max-simple} and perform the exhaustive verification has been provided on GitHub \cite{Code}. Note that if the narrow class number of a real quadratic field is one, the norm of the fundamental unit must be $-1$. This implies that all odd prime factors of $D$ are congruent to
1 modulo 4. Furthermore, if 
$(2)$
 is inert,  the discriminant must satisfy $D \equiv 5 \pmod 8$.   Conversely, if $(2)$ is not inert, then either $2 \mid D$ or $D \equiv 1 \pmod 8$. 
\end{Rmk}
This completes the proof of the first part of Theorem \ref{Thm1}.
\section{The case of Eisenstein series and cusp forms}\label{Section 4}
In this section we assume $F$ is a real quadratic field with narrow class number one and $D>5$, so its fundamental unit $\epsilon_0$  has norm $-1$. Therefore, $M_k(\Gamma_F)=\{0\}$ for odd $k$ via the action of $\epsilon_0I$.  We consider the product of an Eisenstein series $f$ of even weight $k_1$ with a cuspidal eigenform $h$ of even weight $k_2$ and assume $g=f\cdot h$ is also a Hecke eigenform. 

 The proof of the second part of Theorem \ref{Thm1} proceeds in two steps. First, we identify all possible triples $(D,k_1,k_2)$  by analyzing the properties of the Fourier coefficients in the  eigenform product identities. Second, employing the main theorem of \cite{Rankin-Cohen} together with the grand Riemann hypothesis, we rule out those triples corresponding to non-trivial dimensions of the space of cusp forms.

To prove the second part of Theorem \ref{Thm1}, we first bound  the coefficient of Eisenstein series. 
\begin{Lem}\label{EisensteinCoeBound} 
If $\mathfrak m$ is a non-zero integral ideal of $F$, then $|c(\mathfrak m,E_k)|\leq N(\mathfrak m)^{k+1}$.
\end{Lem}

\begin{proof}
The proof follows the approach of \cite[Lemma 3.2]{You-Zhang2021Hilbert}. For any prime ideal $\mathfrak{p}$ and $j \geq 1$, we have 
\begin{align*}
|c(\mathfrak{p}^{j+1},E_k)| 
&= |c(\mathfrak{p}^{j},E_k)c(\mathfrak{p},E_k) - N(\mathfrak{p})^{k-1}c(\mathfrak{p}^{j-1},E_k)| \\
&\leq |c(\mathfrak{p}^{j},E_k)c(\mathfrak{p},E_k)| + |N(\mathfrak{p}^2)^{k-1}c(\mathfrak{p}^{j-1},E_k)|
\end{align*}
by \eqref{primeCoe}  and 
$c(\mathfrak{p},E_k) = 1 + N(\mathfrak{p})^{k-1} \leq 2N(\mathfrak{p})^{k-1}$ by \eqref{EisensteinCoe}. This implies
\[
|c(\mathfrak{p}^{m},E_k)| \leq a_m N(\mathfrak{p}^{m})^{k-1}, \quad m \geq 0,
\]
where $\{a_m\}$ satisfies $a_0=1$, $a_1=2$, and $a_{m+2}=2a_{m+1}+a_m$. Inductively, $a_m \leq 3^m$, yielding
\[
|c(\mathfrak{p}^{m},E_k)|\leq 3^{m}N(\mathfrak{p}^{m})^{k-1}<N(\mathfrak{p}^{2})^{m}\cdot N(\mathfrak{p}^{m})^{k-1}=N(\mathfrak{p}^{m})^{k+1}.
\] The case for non-zero integral ideals $\mathfrak{m}$ follows via coefficient multiplicativity.
\end{proof}
Next we recall the main theorem of \cite{Rankin-Cohen} concerning real quadratic fields and parallel weights, which establishes sufficient conditions for the non-existence of eigenform product identities. 
\begin{Thm}\cite[Main Theorem, Remark 3.4]{Rankin-Cohen}\label{RC2024}
Let $F$ be a real quadratic field of narrow class number one and $k,l\geq 2$ be even integers. Under the grand Riemann hypothesis, if $f=E_k$ is an Eisenstein series and $h \in S_{l}(\Gamma_F)$  a normalized eigenform, then $f\cdot h$ is not an eigenform whenever $\mathrm{dim}\ S_{k+l}(\Gamma_F)> 1$. 
\end{Thm}
\begin{Rmk}
Note that Theorem \ref{RC2024} requires the grand Riemann hypothesis when $k=2$, explaining the formulation of the second part of Theorem \ref{Thm1}.     
\end{Rmk}

With the dimension formula for  Hilbert cusp form spaces over quadratic fields established in \cite{Thomas}, we obtain the following lemma.
\begin{Lem}\label{DimThomas}
Let \( F=\mathbb Q(\sqrt{d}) \) be a real quadratic field with square-free $d$, discriminant \( D >12\) and narrow class number one. For \( k \geq 2 \),
\[
\dim S_{2k}(\Gamma_F) = 2k(k-1)\cdot\zeta_F(-1) + \chi(\Gamma_F) - h(-3D)\cdot \delta_k/6,
\]
where 
$\delta_k$ is $1$ if $k \equiv 2\text{(mod 3)}$ and $0$ otherwise,
$h(D)$ is the class number of quadratic field with discriminant $D$, 
and $\chi(\Gamma_F)=1+\dim S_2(\Gamma_F)$ is the arithmetic genus. 
\end{Lem}
\begin{proof}
The dimension formula follows directly from \cite[Eq (2.15)]{Thomas}. Specifically, \cite[Eq (2.8)]{Thomas} shows that $\sum_{k=0}^{\infty} \delta_k t^k = \frac{t^2}{1 - t^3}$, and hence we get
\[
\sum_{k=0}^{\infty} \delta_k t^k = \frac{t^2}{1 - t^3}
= t^2 \cdot \sum_{n=0}^{\infty} t^{3n} = \sum_{n=0}^{\infty} t^{3n + 2}, \quad |t|<1.
\]
By comparing coefficients of \(t^k\) on both sides, we get that \(\delta_k = 1\) when \(k = 3n + 2\) for some integer \(n \geq 0\) and 
\(\delta_k = 0\) otherwise. 

Since elliptic points of order $5$ do not occur for $D>12$, $a_5(\Gamma)$ in \cite[Eq (2.15)]{Thomas} is zero. Now we consider the contribution of order $3$ elliptic points to the dimension formula. Since the  fundamental unit of $F$ has norm $-1$, implying that the Pell equation $x^2-dy^2=-1$ has solutions. It follows that all odd prime divisors of $D$ are congruent to 1 modulo 4, so $3\nmid D$. Applying the second formula on page 17 of \cite{Geer}, we obtain the desired dimension formula.
\end{proof}

\begin{Cor}\label{Dim}
For $k\geq3$, $\dim S_{2k}(\Gamma_F)>1$ for $D> 12$.
\end{Cor}

\begin{proof}
From the Dirichlet class number formula and \cite[Corollary 1]{Ramare2001L1chi},
we have $$h(\Delta)=\frac{|\Delta|^{1/2}}{\pi}\cdot L(1,\chi)\leq \frac{|\Delta|^{1/2}}{\pi}\cdot\left(\frac{\log (|\Delta|)}{2}+\frac{5}{2}-\log6\right)$$ for fundamental  discriminant $\Delta<-4$, Kronecker symbol $\chi=\left(\frac{\Delta}{\cdot}\right)$ and Dirichlet $L$-series $L(1,\chi)$. By the functional equation  
$\zeta_F(-1) = D^{3/2}\cdot (4\pi^{4})^{-1} \cdot\zeta_F(2)$ and Lemma \ref{DimThomas}, 
\begin{equation}\label{dim-neq}
\begin{aligned}
\dim S_{2k}(\Gamma_F)=&2k(k-1)\cdot\zeta_F(-1) +1+\dim S_2(\Gamma_F)- h(-3D)\cdot \delta_k/6\\
\geq&2k(k-1)\cdot D^{3/2}\cdot (4\pi^4)^{-1}+1-h(-3D)/6\\
\geq&3\cdot  D^{3/2}\cdot(\pi^4)^{-1}+1- \frac{(3D)^{1/2}}{6\pi}\left(\frac{\log (3D)}{2}+\frac{5}{2}-\log6\right)
\end{aligned}
\end{equation}
for $k\geq 3$, since $\zeta_F(2)>1$ and $-3D$ is also a fundamental discriminant.
Note that the right-hand side of \eqref{dim-neq} increases with $D$ and exceeds $1$ for $D > 12$, hence $\dim S_{2k}(\Gamma_F) > 1$.
\end{proof}
We now apply the bounds from \cite{Joshi-Zhang2019Hilbert} to determine all possible triples $(k_1,k_2,D)$. When $(2)$ is inert, the Hecke relation \eqref{primeCoe} for $g=f \cdot h$ yields
\begin{align*}
    &\left|c_0(f)\cdot\frac{c((4),g)}{c_0(f)} -c_0(f)c((4),h)\right|\\
    =&\left|c_0(f)\cdot \left(\left(\frac{c((2),g)}{c_0(f)}\right)^2-4^{k_1+k_2-1}\right)-c_0(f)\cdot \left(c((2),h)^2-4^{k_2-1}\right)\right|\\
    =&\left|c_0(f)^{-1}+ 2c((2),h)+c_0(f)(4^{k_2-1}-4^{k_1+k_2-1})\right| . 
\end{align*}
On the other hand, Lemma 4.3 of \cite{Joshi-Zhang2019Hilbert} implies 
\begin{equation*}
c((4), g)-c_0(f) c( (4), h) =c((3), h) + c(( 3), f) + c(( 2), h)c(( 2), f).
\end{equation*}
Hence, by \eqref{cuspFourierBound}, \eqref{EisensteinCoe}, \eqref{zeta_bound}, Lemma \ref{EisensteinCoeBound}, and \cite[Lemma 3.2]{You-Zhang2021Hilbert}, eigenform product identity $g=f\cdot h$ must satisfy 
\begin{equation}\label{eq:3.4}
\begin{aligned}
&4^{k_2-1}(4^{k_1}-1) \\
\leq&\left(\frac{\pi^5}{18}\right)^{2} \left( \frac{(2\pi)^2}{D} \right)^{2k_1-1} \Gamma(k_1)^{-4} + \frac{\pi^5}{18} \left( \frac{(2\pi)^2}{D} \right)^{k_1-\frac{1}{2}} \Gamma(k_1)^{-2} \cdot 2^{k_2+1}  \\
+& \frac{\pi^5}{18} \left( \frac{(2\pi)^2}{D} \right)^{k_1-\frac{1}{2}} \Gamma(k_1)^{-2} \cdot \left(3^{k_2+3}+9^{k_1+1}+(1+4^{k_1-1})\cdot2^{k_2}\right)\\
\leq& \frac{\pi^5}{6}\left( \frac{(2\pi)^2}{D} \right)^{k_1-\frac{1}{2}} \cdot\Gamma(k_1)^{-2}\cdot \left(3^{k_2+3}+9^{k_1+1}+(1+4^{k_1-1})\cdot2^{k_2}\right).
\end{aligned}
\end{equation}
By removing \(D\), we obtain
\begin{align}\label{eq:3.5}
4^{k_2-1}(4^{k_1}-1)&\leq\frac{\pi^5}{6}\cdot (2\pi)^{2k_1-1}\cdot\Gamma(k_1)^{-2} \cdot \left(3^{k_2+3}+9^{k_1+1}+(1+4^{k_1-1})\cdot2^{k_2}\right).
\end{align}
By dividing both sides of \eqref{eq:3.5} by \(4^ {k_2}\), we have 
\begin{align}\label{eq:3.6}
4^{k_1}-1 
&\leq\frac{2\pi^5}{3}\cdot (2\pi)^{2k_1-1}\cdot\Gamma(k_1)^{-2} \cdot \left(28+9^{k_1+1}+4^{k_1-1}\right).
\end{align}

 If $(2)$ is not inert, subsections 6.1 and 6.2 of \cite{Joshi-Zhang2019Hilbert} imply that $g=f\cdot h$ satisfies 
\begin{align}\label{eq:3.4.1}
2^{k_2-1}(2^{k_1}-1) 
&\leq \frac{\pi^5}{18}\cdot \left( \frac{(2\pi)^2}{D} \right)^{k_1-\frac{1}{2}}\cdot \Gamma(k_1)^{-2},
\end{align}
\begin{equation}\label{eq:3.5.1}
2^{k_2-1}(2^{k_1}-1) 
\leq \frac{\pi^5}{18}\cdot (2\pi)^{2k_1-1}\cdot \Gamma(k_1)^{-2},
\end{equation}
and 
\begin{equation}\label{eq:3.6.1}
2^{k_1}-1 
\leq \frac{\pi^5}{18}\cdot (2\pi)^{2k_1-1}\cdot \Gamma(k_1)^{-2}.
\end{equation}

Now we prove the second part of Theorem \ref{Thm1} according to inert and non-inert cases in ideal (2). 
\begin{Prop}\label{(2)Inert}
Under the grand Riemann hypothesis, no eigenform product identity $g=f\cdot h$ with $c_0(f)\neq0$  and $c_0(h)=0$ exists over all real quadratic fields with narrow class number 1, $D>5$ and $(2)$ inert.
\end{Prop}
\begin{proof}
Assume that $(2)$ is inert, we begin by analyzing the growth of the right-hand side of inequality \eqref{eq:3.6} in $k_1$. Define  $$G(k_1)=\frac{2\pi^5}{3}\cdot (2\pi)^{2k_1-1}\cdot\Gamma(k_1)^{-2} \cdot(28+9^{k_1+1}+4^{k_1-1})$$ and we have
\begin{align*}
\frac{G(k_1)}{G(k_1-1)}=\frac{(2\pi)^{2k_1-1}\cdot\Gamma(k_1)^{-2} \cdot(28+9^{k_1+1}+4^{k_1-1})}{(2\pi)^{2k_1-3}\cdot\Gamma(k_1-1)^{-2} \cdot(28+9^{k_1}+4^{k_1-2})}&<\frac{(2\pi)^2\cdot9 }{(k_1-1)^2}<1 
\end{align*}
for $k_1\geq20$. Hence $G$ decreases over $S = \{20 + 2k \mid k \in \mathbb{N}\}$. 
 By the monotonicity of $G(k_1)$, we find that inequality \eqref{eq:3.6} fails for $k_1 > 28$, establishing $28$ as the maximal $k_1$. As both sides of inequality \eqref{eq:3.5} increase strictly in $k_2$, but at different rate, we determine the maximal $k_2$ for each $k_1$ by iterating over $k_2 = 2, 4, 6, \cdots$.
Similarly, for each fixed pair $(k_1, k_2)$, \eqref{eq:3.4} yields the maximal $D$ (see Table \ref{E-S-(2) inert}). 

By Corollary \ref{Dim}, $\dim S_{k_1 + k_2}(\Gamma_F) > 1$ for all possible triples $(k_1,k_2,D)$ where $k_1+k_2\geq 6$ and $D\geq 13$, and hence  no eigenform product identity exists in this case by Theorem \ref{RC2024} and Table \ref{E-S-(2) inert}.

It remains to consider $k_1 = k_2=2$ and $D\in[13,3517]$. Following the lines of Corollary \ref{Dim}, $\dim S_k(\Gamma_F) > 1$ for $k=4$ and $D \geq 29$, reducing the remaining case to $(k_1, k_2, D) = (2, 2, 13)$. However, \cite{Ishikawa1988dim} shows  no weight-$2$ cuspidal eigenform exists in this case. Therefore, there is no eigenform product identity over all real quadratic fields  with narrow class number 1, $D>5$ and $(2)$ inert.
\end{proof}
\begin{table}[ht]
  \centering
  \caption{The possible $k_1$, $k_2$ and $D$}
  \label{E-S-(2) inert}
  \begin{tabular}{|c|c|c|c|c|c|c|c|c|c|c|c|c|c|c|}
    \hline
    $k_1$ & 2 &4&6&8&10&12&14&16&18&20&22&24&26&28 \\\hline
    \text{Maximal} $k_2$  & 38  &42 &38&26&18&16&16&14&14&12&8&6&4&$\emptyset$\\\hline
    \text{Maximal $D$}&3517&109&37&13&$\emptyset$&$\emptyset$&$\emptyset$&$\emptyset$&$\emptyset$&$\emptyset$&$\emptyset$&$\emptyset$&$\emptyset$&$\emptyset$\\\hline
  \end{tabular}
\end{table}
\begin{Prop}\label{(2)Non-inert}
Under the grand Riemann hypothesis, no eigenform product identity $g=f\cdot h$ with $c_0(f)\neq0$ and $c_0(h)=0$ exists over all real quadratic fields with narrow class number 1, $D>5$ and $(2)$ non-inert. 
\end{Prop}
\begin{proof}
Assume that $(2)$ is not inert, we put $$G(k_1)=\frac{\pi^5}{18}\cdot (2\pi)^{2k_1-1}\cdot\Gamma(k_1)^{-2}$$ and  have 
\begin{align*}
\frac{G(k_1)}{G(k_1-1)}&=\frac{(2\pi)^{2k_1-1}\cdot\Gamma(k_1)^{-2} }{(2\pi)^{2k_1-3}\cdot\Gamma(k_1-1)^{-2} }=\frac{(2\pi)^2 }{(k_1-1)^2}<1 
\end{align*}
when $k_1\geq8$. Hence $G$ strictly decreases in $S = \{8 + 2k \mid k \in \mathbb{N}\}$ and then inequality \eqref{eq:3.6.1} implies that the maximal possible value of $k_1$ is $12$. For each possible $k_1$, we then determine the corresponding maximal  $k_2$ satisfying  inequality \eqref{eq:3.5.1}. Finally, for every fixed pair $(k_1,k_2)$, we compute the maximal possible $D$ satisfying inequality \eqref{eq:3.4.1}, yielding  Table \ref{E-S-(2) not-inert}.

This reduces to $k_1=2 \text{ or }4$ by Table \ref{E-S-(2) not-inert}.
For $k_1=4$, Table \ref{E-S-(2) not-inert} shows that the only possible quadratic field is $F=\mathbb Q(\sqrt 8)$. In this case, Magma computation yields $\dim S_{4+k_2}(\Gamma_F)>1$ for $k_2\in [2,14]$ and it follows that no eigenform product identity exists by Theorem \ref{RC2024}.  For $k_1=2$, only the triple $(k_1,k_2,D)=(2,2,8)$ yields $\dim S_{4}(\Gamma_F)=1$. However, no weight-$2$ cuspidal eigenform exists for $D=8$ by \cite{Ishikawa1988dim}. Thus, no eigenform product identity exists over all real quadratic fields with narrow class number 1, $D>5$ and $(2)$ non-inert.     
\end{proof}
\begin{table}[ht]
  \centering
  \caption{The possible $k_1$, $k_2$ and $D$}
  \label{E-S-(2) not-inert}
  \begin{tabular}{|c|c|c|c|c|c|c|}
    \hline
    \text{$k_1$} &2 &4&6&8&10&12 \\\hline
    \text{Maximal $k_2$}&10&14&14&12&8&2 \\\hline
    \text{Maximal $D$}&73&8&$\emptyset$&$\emptyset$&$\emptyset$&$\emptyset$\\\hline
  \end{tabular}
\end{table}
\begin{Rmk}
The dimensions of the associated cusp form spaces were computed using the online Magma calculator, while the data in Tables \ref{E-S-(2) inert} and \ref{E-S-(2) not-inert} were generated using SageMath.  All corresponding code is available in our GitHub repository~\cite{Code}.
\end{Rmk}
We complete the proof of the second part of Theorem \ref{Thm1}.

\section{On totally real number fields beyond the quadratic case}\label{Section 5}
In this section, we prove Theorem \ref{Thm2} by considering two Eisenstein series of distinct weights. This theorem provides an analogous result for all totally real number fields of degree greater than $2$ and with arbitrary narrow class number. Following our earlier approach, we apply the  bounds for special values of Hecke $L$-series and Odlyzko's bound to demonstrate that these estimates contradict the relations on the constant terms of Eisenstein series derived from Hecke eigenform product identity. 

Let $F$ be a totally real number field of degree $n>2$, and let $f=E_{k_1}(\phi_1,\psi_1)$ and $h=E_{k_2}(\phi_2,\psi_2)$, where $k_1\neq k_2$ and $\phi_i$, $\psi_i$ are  narrow ideal class characters of $F$ (see \cite[\S 2]{Dasgupta2011Hilbert} for details). For a narrow ideal class character $\phi$, we recall the following estimate for Hecke $L$-function $L(s,\phi)$ (see \cite[Eq (2.4)]{You-Zhang2021Hilbert}):
\begin{equation}\label{Lbound}
\begin{aligned}
     \left(\frac{2}{\pi} \right)^{\frac{n}{2}}\left( \frac{D}{\left( 2\pi\right)^{n}}\right)^{k-\frac{1}{2}}\Gamma\left( k\right)^{n}\frac{\zeta\left( n^{2}k\right)}{\zeta\left( k\right)^{n}} &\leq \left| L\left( 1-k,\phi\right) \right|\\
     &\leq  \left(\frac{2}{\pi} \right)^{\frac{n}{2}}\left( \frac{D}{\left( 2\pi\right)^{n}}\right)^{k-\frac{1}{2}}\Gamma\left( k\right)^{n}\zeta\left( k \right)^{n}. 
\end{aligned}
\end{equation}

As in \cite[Eq (2.1) and Eq (2.3)]{You-Zhang2021Hilbert}, $g$ is equal to $E_{k_1+k_2}(\phi,\psi)$ up to a non-zero scalar, and the identity $g=f\cdot h$ implies \begin{equation} \label{eq:Eisensteinconstantterm}
    \frac{1}{L\left( 1-k_{1},\phi_{1}^{-1} \psi_{1}\right)}+ \frac{1}{L\left( 1-k_{2},\phi_{2}^{-1}\psi_{2}\right)}=\frac{1}{L\left( 1-k_{1}-k_{2},\phi^{-1}\psi\right)},
\end{equation} 
where $\phi=\phi_1\phi_2$ and $\psi=\psi_1\psi_2$. Put $A=L\left( 1-k_{1},\phi_{1}^{-1} \psi_{1}\right)$, $B=L\left( 1-k_{2},\phi_{2}^{-1}\psi_{2}\right)$ and $C=L\left( 1-k_{1}-k_{2},\phi^{-1}\psi\right)$.

Assume that $k_1>k_2$. By 
\eqref{Lbound}, we obtain  \begin{align} \label{eq:unequalweightlowerbound1}
    \left| \frac{A}{B} \right| \geq \left( \frac{D}{\left( 2\pi\right)^{n}}\right)^{k_{1}-k_{2}}\frac{\Gamma\left( k_{1}\right)^{n}}{\Gamma\left( k_{2}\right)^{n}}\frac{\zeta\left( n^{2}k_{1}\right)}{\zeta\left( k_{1}\right)^{n}\zeta\left( k_{2}\right)^{n}}\geq\left( \frac{Dk_2^n}{\left( 2\pi\right)^{n}}\right)^{k_{1}-k_{2}}\frac{1}{\zeta(2)^{2n}}.
\end{align} 
Similarly, we obtain that
\begin{align}\label{1inequality}
    1=\left| \left( A+B\right) \frac{C}{AB} \right| = \left| \frac{A}{B}+1 \right| \cdot \left| \frac{C}{A}\right|\geq  C_{n} \left( \frac{D}{\left( 2\pi\right)^{n}}\right)^{k_{2}}\frac{\Gamma\left( k_{1}+k_{2}\right)^{n}}{\Gamma\left( k_{1}\right)^{n}},
\end{align} where $C_{n}=\delta_{n} \cdot \frac{\zeta\left( n^{2}\left(k_{1} +k_{2}\right)\right)}{\zeta\left( k_{1}+k_{2}\right)^{n}\zeta\left( k_{1}\right)^{n}}$ and $\delta_{n}$ arises from estimating $\left|\frac{A}{B}+1\right|$. We shall prove $\delta_n$ can be chosen as a constant  depending only on $n$. To proceed, we recall  Odlyzko's bound from ~\cite{TAKEUCHI83}. 
\begin{Prop}{\cite[Proposition 2.3]{TAKEUCHI83}}\label{Takeuchi's bound}
    We have $$D>a^{n}\exp(-b),$$ where $a=29.099$ and $b=8.3185$.
\end{Prop}
\begin{Lem}\label{deltaN}
    We can take $\delta_n = 
\begin{cases} 
0.2 & \text{if } n = 3, \\
0.03 & \text{if } n = 4, \\
1 & \text{if } n > 4.
\end{cases}$  
\end{Lem}

\begin{proof}
We first consider $n>4$ and it suffices to prove $\left| \frac{A}{B} \right| \geq 2$. In this case, by \eqref{eq:unequalweightlowerbound1} and Proposition \ref{Takeuchi's bound}, we find that for each fixed $n>4$, 
\begin{align}\label{AB}
    \left| \frac{A}{B} \right| \geq \frac{\left(\left(\frac{ak_2}{2\pi}\right)^{n}\exp(-b)\right)^{k_1-k_2}}{ \zeta(2)^{2n}}\geq\left(\frac{6a}{\pi^3}\right)^{2n}\cdot \exp(-2b),
    \end{align}
     since $\left(\left(\frac{ak_2}{2\pi}\right)^{n}\exp(-b)\right)^{k_1-k_2}$ is minimal when $k_{1}=4$ and $k_{2}=2$. As the function $$\left(\frac{6a}{\pi^3}\right)^{2n}\cdot \exp(-2b)$$ increases with $n$, it follows that  $\left| \frac{A}{B} \right| \geq 2$, and hence $\left| \frac{A}{B}+1 \right| \geq 1$ for all $n\geq 6$. For $n=5$, substituting $D\geq14641$ from \cite[Table 3]{Voight} into the stronger bound of \eqref{eq:unequalweightlowerbound1} also yields  $\left| \frac{A}{B}+1 \right| \geq 1$. 
    
    For $n=3,4$, we use the minimal discriminants $49$ and $725$   respectively from \cite[Table 3]{Voight} to  bound $\left| \frac{A}{B}+1 \right|$. For $n=3$, the expression $$\left(\frac{49\cdot k_2^n}{(2\pi)^n}\right)^{k_1-k_2} \zeta(2)^{-2n} \approx 0.786299$$ is closest to $1$ when $k_{2}=2$ and $k_1=8$, yielding  $\left| \frac{A}{B}+1 \right| \geq 0.2$. For $n=4$, the expression  $$\left(\frac{725\cdot k_2^n}{(2\pi)^n}\right)^{k_1-k_2} \zeta(2)^{-2n} \approx 1.033449$$ is closest to $1$ when $k_{2}=2$ and  $k_{1}=4$, giving  $\left| \frac{A}{B}+1 \right| \geq 0.03$.
\end{proof}

Now we verify whether \eqref{1inequality} holds by Proposition \ref{Takeuchi's bound}  and Lemma \ref{deltaN}  to prove Theorem \ref{Thm2}. 
\begin{Prop}\label{section4Nonequal}
No eigenform product identity $g=E_{k_{1}}\left( \phi_{1},\psi_{1}\right)\cdot E_{k_{2}}\left( \phi_{2},\psi_{2}\right)$ with $k_{1} > k_{2}$  exists over all totally real number fields  of degree $n > 2$.
\end{Prop}
\begin{proof}
To complete the proof, it suffices to prove \eqref{1inequality} does not hold. We first consider $n=3$. Since the term $$ \frac{\zeta\left( 9\left(k_{1} +k_{2}\right)\right)}{\zeta\left( k_{1}+k_{2}\right)^{3}\zeta\left( k_{1}\right)^{3}}\geq \frac{1}{\zeta\left( 6\right)^{3}\cdot\zeta\left( 4\right)^{3}} \geq 0.74,$$ Lemma \ref{deltaN} gives $C_{3} \geq 0.14$. Since  $k_{1}\geq4$, we have  
$$C_{3} \left( \frac{D}{\left( 2\pi\right)^{3}}\right)^{k_{2}}\frac{\Gamma\left( k_{1}+k_{2}\right)^{3}}{\Gamma\left( k_{1}\right)^{3}}\geq C_3\left( \frac{Dk_{1}^{3}}{\left( 2\pi\right)^{3}}\right)^{k_{2}}\geq 0.14\cdot \left(\frac{49\cdot4^3}{\left( 2\pi\right)^{3}}\right)^2> 22.37,$$
 which contradicts \eqref{1inequality}.

Then we consider $n=4$ and we see that $C_{4}\geq \delta_{4} \cdot\frac{1}{\zeta\left( 6\right)^{4}\zeta\left( 4\right)^{4}} \geq 0.02$, while the term $$\left( \frac{D}{\left( 2\pi\right)^{4}}\right)^{k_{2}}\frac{\Gamma\left( k_{1}+k_{2}\right)^{4}}{\Gamma\left( k_{1}\right)^{4}}\geq \left( \frac{725\cdot 4^{4}}{\left( 2\pi\right)^{4}}\right)^{2} \geq 14181,$$ so \eqref{1inequality} does not hold for $n=4$.

It remains to consider $n>4$. 
For each fixed $n>4$, note that $$\left(\left(\frac{a\cdot k_1}{2\pi}\right)^{n}\exp(-b)\right)^{k_2}$$ increases with respect to $k_1$ and $k_2$, and  $$\left(\left(\frac{2a}{\pi}\right)\cdot \frac{90}{\pi^{4}}\right)^{2n}\cdot \exp(-2b)$$ increases with respect to $n$. Hence the expression  $$\left(\frac{90}{\pi^4}\right)^{2n}\left(\left(\frac{a\cdot k_1}{2\pi}\right)^n\cdot \exp(-b)\right)^{k_2}$$ achieves its minimum  at $(n,k_1,k_2)=(5,4,2)$, that is  
\begin{align*}
    C_{n} \left( \frac{D}{\left( 2\pi\right)^{n}}\right)^{k_{2}}\frac{\Gamma\left( k_{1}+k_{2}\right)^{n}}{\Gamma\left( k_{1}\right)^{n}}
    &\geq \left(\frac{180a}{\pi^5}\right)^{2n}\cdot \exp(-2b)>128426,
\end{align*}
contradicting \eqref{1inequality} for all $n > 4$.  This completes the proof.
\end{proof}
By Proposition \ref{section4Nonequal}, we finish the proof of Theorem \ref{Thm2}.

\bibliographystyle{amsplain}
%\bibliography{sample}

\end{document}